\documentclass{amsart}
\usepackage{amssymb,mathrsfs,MnSymbol}
\usepackage{color}

\def\bN {\mathbf{N}}

\def\bR {\mathbf{R}}

\def\fH {\mathfrak{H}}

\def\cC {\mathcal{C}}
\def\cD {\mathcal{D}}
\def\cE {\mathcal{E}}

\def\cH {\mathcal{H}}

\def\cL {\mathcal{L}}

\def\de {{\delta}}
\def\eps {{\epsilon}}

\def\l {{\lambda}}

\def\Om {{\Omega}}

\def\d {{\partial}}
\def\grad {{\nabla}}
\def\Dlt {{\Delta}}

\def\rstr {{\big |}}
\def\indc {{\bf 1}}

\def\la {\langle}
\def\ra {\rangle}
\def \La {\bigg\langle}
\def \Ra {\bigg\rangle}

\newcommand{\Supp}{\operatorname{supp}}

\newcommand{\Tr}{\operatorname{trace}}

\newcommand{\Lip}{\operatorname{Lip}}

\def\OP {\mathrm{OP}}

\newcommand{\MKd}{\operatorname{dist_{MK,2}}}

\def\hb {{\hbar}}

\newcommand{\ba}{\begin{aligned}}
\newcommand{\ea}{\end{aligned}}

\newcommand{\be}{\begin{equation}}
\newcommand{\ee}{\end{equation}}

\newcommand{\lb}{\label}

\newcommand{\nn}{\nonumber}
\newcommand{\hus}[1]{\widetilde W_{\hbar}[#1]}

\renewcommand{\tilde}{\widetilde}

\newtheorem{Thm}{Theorem}[section]

\newtheorem{Prop}[Thm]{Proposition}
\newtheorem{Cor}[Thm]{Corollary}
\newtheorem{Lem}[Thm]{Lemma}

\begin{document}

\title[Observability for Schr\"odinger]{Observability for the Schr\"odinger Equation:\\ an Optimal Transportation Approach}

\author[F. Golse]{Fran\c cois Golse}
\address[F.G.]{CMLS, \'Ecole polytechnique, CNRS, Universit\'e Paris-Saclay , 91128 Palaiseau Cedex, France}
\email{francois.golse@polytechnique.edu}

\author[T. Paul]{Thierry Paul}
\address[T.P.]{Laboratoire J.-L. Lions, Sorbonne Universit\'e \& CNRS, bo\^\i te courrier 187, 75252 Paris Cedex 05, France}
\email{thierry.paul@upmc.fr}

\begin{abstract}
We establish an observation inequality for the Schr\"odinger equation on $\bR^d$, uniform in the Planck constant $\hbar\in[0,1]$. The proof   is based on the pseudometric introduced in [F. Golse, T. Paul, Arch. Rational Mech. Anal. \textbf{223} (2017), 57--94].
This inequality involves only effective constants which are computed explicitly in their dependence in  $\hbar$ and all parameters involved.
\end{abstract}

\maketitle

\section{Observation inequality for the Schr\"odinger equation}


Consider the Schr\"odinger equation where the (real-valued) potential $V$ belongs to $C^{1,1}(\bR^d)$ is such that the quantum Hamiltonian
\be\label{schro}\nn
-\tfrac12\hbar^2\Dlt_y+V(y)
\ee 
has a self-adjoint extension on $\fH:=L^2(\bR^d)$:
\be\lb{Schro}
i\hbar\d_t\psi(t,y)=\left(-\tfrac12\hbar^2\Dlt_y+V(y)\right)\psi(t,y)\,,\qquad\psi\rstr_{t=0}=\psi^{in}\,.
\ee
In the equation above, $\hbar>0$ the reduced Planck constant, and the particle mass is set to $1$.

An observation inequality for the Schr\"odinger equation \eqref{Schro} is an inequality of the form
\be\label{ineqobs}
\|\psi^{in}\|_{\fH}^2\le C
\int_0^T\int_\Om|\psi(t,x)|^2dxdt\,,
\ee
for some $T>0$, where $\Om$ is an open subset of $\bR^d$, and $C\equiv C[T,\Om]$ is a positive constant, which holds for some appropriate class of initial data $\psi^{in}$   (see equation
(2) in \cite{Laurent}).

\noindent Note that the r.h.s. of \eqref{ineqobs} is smaller than $CT$, so that \eqref{ineqobs} can be satisfied only when $CT\geq 1$. Moreover, it is easy to check that the case $CT=1$ is possible only when $\Omega=\bR^d$, and reduces that way to a tautology. 

Therefore we will suppose in the sequel
$$
CT>1.
$$

We will say that    a compact subset $K$ of $\bR^d\times\bR^d$,  an open set $\Om$ of $\bR^d$ and  $T>0$ satisfy the ``(\`a la) Bardos-Lebeau-Rauch geometric condition'' \cite{BLR}  if:
$$
\hbox{for each }(x,\xi)\in K\,,\hbox{ there exists }t\in(0,T)\hbox{ s.t. }X(t;x,\xi)\in\Om\,.\leqno{\hbox{(GC)}}
$$

\vskip 1cm
Let us recall the definition of the Schr\"odinger coherent state:
$$
|q,p\ra(x):=(\pi\hbar)^{-d/4}e^{-|x-q|^2/2\hbar}e^{ip\cdot(x-q/2)/\hbar}\,
$$
providing a decomposition of the identity  on $\fH$ (in a weak sense)
\be\label{decomp}
\int_{\bR^{2d}}|q,p\ra\la q,p|\tfrac{dpdq}{(2\pi\hbar)^d}=I_\fH.
\ee
Let us recall also,  for any self-adjoint operator $A$ on $ L^2(\bR^d)$ and any $\psi\in L^2(\bR^d)$, 
the definition of the standard deviation of $A$ in the state $\psi$, $\Delta_A(\psi)\in[0,+\infty]$:
\be\label{defdev}\nn
\Delta_A(\psi)=\sqrt{(\psi,A^2\psi)_{L^2(\bR^d)}-(\psi,A\psi))_{L^2(\bR^d)}^2}
\ee
We define 
\be\label{defdelta}
\Delta(\psi):=
\sqrt{\sum_{j=1}^d(\Delta^2_{x_j}(\psi)+\Delta^2_{-i\hbar\partial_{x_j}}(\psi))}.
\ee
Let us remark that, by the Heisenberg inequalities, for any $\psi\in\fH$,
\be\label{deltaplus}
\Delta(\psi)^2\geq d\hbar.
\ee
and, for any $(p,q)\in\bR^{2d}$,
\be\label{heiscoh}
\Delta(|p,q\rangle)=\sqrt{d\hbar}.
\ee
\vskip 1cm
\begin{Thm}\label{main}
Assume that $V$ belongs to $C^{1,1}(\bR^d)$ and that $V^-\in L^{d/2}(\bR^d)$.

Let $T>0$, $\Omega$ be an open subset of $\bR^d$. 
and $K$ be a compact set in $\bR^{2d}$ satisfying the Bardos-Lebeau-Rauch condition  $(GC)$.

 Moreover, let $\delta>0$ and 
$$
\Omega_\delta:=\{x\in\bR^d|\mathrm{dist}(x,\Omega)<\delta\}.
$$

Then the Schr\"odinger equation \eqref{Schro} satisfies an observability property on $[0,T]\times \Omega_\delta$ of the form \eqref{ineqobs} with constant $C
$ for all vectors $\psi\in\fH$ satisfying
$$
C[T,K,\Omega
]\left(\int_
K|\langle\psi|p,q\rangle|^2\tfrac{dpdq}{(2\pi\hbar)^d}\right)-D[T,\Lip(\nabla V)]\frac{\Delta(\psi)}\delta\geq \frac1{C}
$$
where
\begin{eqnarray}
C[T,K,\Omega
]&=&
\inf_{(x,\xi)\in K}\int_0^T\indc_{\Om
}(X(t;x,\xi))dt
\nn\\
D[T,\Lip(\nabla V)]&=&\frac{e^{(1+\Lip(\nabla V)^2)T/2}-1}{\mbox{\small $1+\Lip(\nabla V)^2$}}.
\nn
\end{eqnarray}
Moreover, the observation inequality will be satisfied for a non empty set of vectors as soon as $\delta$ satisfies the following, non sharp, bound:
$$
\delta\geq \frac{D[T,\Lip(\nabla V)]}{C[T,K,\Omega](1-e^{-\frac{d_K^2}{4\hbar}}/(4\pi)^d)+C^{-1}}\sqrt{d\hbar},
$$
where $d_K$ is the diameter of $K$.
\end{Thm}
\vskip 1cm

The first part of Theorem \ref{main} is exactly the second part (pure state case) of Corollary \ref{C-Obs} of Theorem \ref{T-Obs} in Section \ref{observmet} below.

Controllability of the quantum dynamics has a long history in mathematics and mathematical physics. Giving an exhaustive bibliography on the subject is by far beyond the scope of the present paper, paper, but the reader can consult the survey article
\cite{Laurent} and the literature cited there, together with the important earlier references
 \cite{fabre1,fabre2}, \cite{lebeau1}

For the bound on $\delta$, we first remark that the quantity
$$
E[\psi,\delta]:=C[T,K,\Omega^\delta]\left(\int_
K|\langle\psi|p,q\rangle|^2\tfrac{dpdq}{(2\pi\hbar)^d}\right)-D[T,\Lip(\nabla V)]\frac{\Delta(\psi)}\delta,
$$ 
needed to be strictly positive for the observability condition to hold true, is a difference between (a quantity proportional to)
$\int_
K|\langle\psi|p,q\rangle|^2\tfrac{dpdq}{(2\pi\hbar)^d}$ ($\leq 1$ by \eqref{decomp})\footnote{Note that $\int_{\bR^{2d}}K|\langle\psi|p,q\rangle|^2\tfrac{dpdq}{(2\pi\hbar)^d}$ is the integral over $K$ of the Husimi function of $\psi$.}, which evaluates the microlocalization of $\psi$ on $K$, and (a quantity proportional to) $\Delta(\psi)$ ($\geq \sqrt{d\hbar}$ (by \eqref{deltaplus}), which measures the spreading of $\psi$ near its average position in phase-space.

\noindent However, this competition is balanced by the 
 smallness of $D[T,\Lip(\nabla V)]\frac{\Delta(\psi)}\delta$ for large values of $\delta$, namely $E[\delta,\psi]\geq \frac1C$ when
 $$
\delta\geq \frac{D[T,\Lip(\nabla V)]\Delta(\psi)}{C[T,K,\Omega](1-
 \int_
K|\langle \psi|p,q\rangle|^2\tfrac{dpdq}{(2\pi\hbar)^d}
)+C^{-1}}.
 $$

\noindent Finally, we remark that, taking $\psi=|p_0,q_0\rangle$ for some $(p_0,q_0)\in\bR^{2d}$ we have, by \eqref{heiscoh}, 
$$
\Delta(\psi)=\sqrt{d\hbar},
$$ 
and, when $(p_0,q_0)$ belongs to the interior of $ K$, 
$$
 \int_
K|\langle p_0,q_0|p,q\rangle|^2\tfrac{dpdq}{(2\pi\hbar)^d}=
1-\int_{\bR^{2d}\backslash K}e^{-\frac{|p_0-p|^2+|q_0-q|^2}\hb}
\tfrac{dpdq}{(2\pi\hbar)^d}
\geq 1-\tfrac{e^{-\frac{\mathrm{dist}{\left((p_0,q_0),\bR^{2d}\backslash K\right)}^2}\hb}}{(4\pi)^{d}}.
$$
We conclude by picking  $(p_0,q_0)$ such that, for example, $\mathrm{dist}{\left((p_0,q_0),\bR^{2d}\backslash K\right)}\geq \frac{d_K}2$.

\vskip 1cm

In the present paper, we will be working with the slightly more general Heisenberg equation
\be\lb{Heisen}
i\hbar\d_tR(t)=\left[-\tfrac12\hbar^2\Dlt_y+V(y),R(t)\right]\,,\ R\rstr_{t=0}=R^{in}\geq 0,\ \Tr R=1,\
\ee
equivalent to the Schr\"odinger equation, modulo a global phase of the wave function, through the passage
$$
\psi\in\fH\ \longrightarrow |\psi\rangle\langle\psi|,
$$
and whose underlying classical dynamics solves  the Liouville equation
$$
\d_tf(t,x,\xi)+\{\tfrac12|\xi|^2+V(x),f(t,x,\xi)\}=0\,,\qquad f\rstr_{t=0}=f^{in}\,,
$$
where $f^{in}$ is a probability density on $\bR^d\times\bR^d$ having finite second moments.

Corollary \ref{C-Obs}  contains also an equivalent statement for initial conditions  which are T\"oplitz operators. The general case of mixed states  can be recovered by the inequality \eqref{obsgen} inside the proof of  Theorem \ref{T-Obs}.

The core of the paper is  Theorem \ref{T-Obs} in Section \ref{observmet}, whose proof needs the introduction in Section {pseudomet} of a class of pseudometrics adapted to the Heisenberg equation \eqref{Heisen}, introduced in \cite{FGTPaul} after \cite{FGMouPaul}, and whose evolution under  \eqref{Heisen} is presented in Section \ref{evolmet}.

\section{A pseudometric for comparing classical and quantum densities}\label{pseudomet}


This section elaborates on \cite{FGTPaul}, with some marginal improvements.

A density operator on $\fH$ is an operator $R\in\cL(\fH)$ such that
$$
R=R^*\ge 0\,,\quad\Tr(R)=1\,.
$$
The set of all density operators on $\fH$ will be denoted by $\cD(\fH)$. We denote by $\cD^2(\fH)$ the set of density operators on $\fH$ such that
\be\lb{FiniteEnerg}
\Tr(R^{1/2}(-\hbar^2\Dlt_y+|y|^2)R^{1/2})<\infty\,.
\ee
If $R\in\cD^2(\fH)$, one has
\be\lb{FiniteEnerg2}
\Tr((-\hbar^2\Dlt_y+|y|^2)^{1/2}R(-\hbar^2\Dlt_y+|y|^2)^{1/2})=\Tr(R^{1/2}(-\hbar^2\Dlt_y+|y|^2)R^{1/2})<\infty
\ee
as can be seen from the lemma below (applied to $A=\l^2|y|^2-\hbar^2\Dlt_y$ and $T=R$).

\begin{Lem}\lb{L-FiniteEnerg}
Let $T\in\cL(\fH)$ satisfy $T=T^*\ge 0$, and let $A$ be an unbounded operator on $\fH$ such that $A=A^*\ge 0$. Then 
$$
\Tr(T^{1/2}AT^{1/2})=\Tr(A^{1/2}RA^{1/2})\in[0,+\infty]\,.
$$
\end{Lem}

\begin{proof}
The definition of $T^{1/2}$ and $A^{1/2}$ can be found in Theorem 3.35 in chapter V, \S 3 of \cite{Kato}, together with the fact that $A^{1/2}$ and $T^{1/2}$ are self-adjoint. 

If $\Tr(T^{1/2}AT^{1/2})<\infty$, then $A^{1/2}T^{1/2}\in\cL^2(\fH)$ and the equality holds by formula (1.26) in chapter X, \S 1 of \cite{Kato}. If $\Tr(T^{1/2}AT^{1/2})=\infty$, then $\Tr(A^{1/2}TA^{1/2})=+\infty$, for otherwise $T^{1/2}A^{1/2}$
and its adjoint $A^{1/2}T^{1/2}$ would belong to $\cL^2(\fH)$, so that $T^{1/2}AT^{1/2}\in\cL^1(\fH)$, which would be in contradiction with the assumption that $\Tr(T^{1/2}AT^{1/2})=\infty$.
\end{proof}

Let $f\equiv f(x,\xi)$ be a probability density on $\bR^d\times\bR^d$ such that
\be\lb{2ndMom}
\iint_{\bR^d\times\bR^d}(|x|^2+|\xi|^2)f(x,\xi)dxd\xi<\infty\,.
\ee

A coupling of $f$ and $R$ is a measurable operator-valued function $(x,\xi)\mapsto Q(x,\xi)$ such that, for a.e. $(x,\xi)\in\bR^d\times\bR^d$,
$$
Q(x,\xi)=Q(x,\xi)^*\ge 0\,,\quad\Tr(Q(x,\xi))=f(x,\xi)\,,\quad\iint_{\bR^d\times\bR^d}Q(x,\xi)dxd\xi=R\,.
$$
The second condition above implies that $Q(x,\xi)\in\cL^1(\fH)$ for a.e. $(x,\xi)\in\bR^d\times\bR^d$. Since $\cL^1(\fH)$ is separable, the notion of strong and weak measurability are equivalent for $Q$. The set of couplings of $f$ and $R$ is
denoted by $\cC(f,R)$. Notice that the function $(x,\xi)\mapsto f(x,\xi)R$ belongs to $\cC(f,R)$.

In \cite{FGTPaul}, one considers the following ``pseudometric'': for each probability density $f$ on $\bR^d\times\bR^d$ and each $R\in\cD^2(\fH)$,
$$
E_{\hbar,\l}(f,R):=\inf_{Q\in\cC(f,R)}\left(\iint_{\bR^d\times\bR^d}\Tr_\fH(Q(x,\xi)^{1/2}c_\l(x,\xi,y,\hbar D_y)Q(x,\xi)^{1/2})dxd\xi\right)^{1/2}
$$
where the quantum transportation cost is the quadratic differential operator in $y$, parametrized by $(x,\xi)\in\bR^d\times\bR^d$:
$$
c_\l(x,\xi,y,\hbar D_y):=\l^2|x-y|^2+|\xi-\hbar D_y|^2\,,\quad D_y:=-i\grad_y\,.
$$

\begin{Lem}
If $R\in\cD^2(\fH)$ while $f$ is a probability density on $\bR^d\times\bR^d$ with finite second moment \eqref{2ndMom}, one has
$$
\ba
\iint_{\bR^d\times\bR^d}\Tr_\fH(Q(x,\xi)^{1/2}c(x,\xi,y,\hbar D_y)Q(x,\xi)^{1/2})dxd\xi
\\
=\iint_{\bR^d\times\bR^d}\Tr_\fH(c(x,\xi,y,\hbar D_y)^{1/2}Q(x,\xi)c(x,\xi,y,\hbar D_y)^{1/2})dxd\xi
\\
\le2\iint_{\bR^d\times\bR^d}(\l^2|x|^2+|\xi|^2)f(x,\xi)dxd\xi+2\Tr(R^{1/2}(-\hbar^2\Dlt_y+\l^2|y|^2)R^{1/2})&<\infty\,
\ea
$$
for each $Q\in\cC(f,R)$.
\end{Lem}

\begin{proof}
Notice that
$$
c_\l(x,\xi,y,\hbar D_y)\le 2\l^2(|x|^2+|y|^2)+2(|\xi|^2-\hbar^2\Dlt_y)=2(\l^2|x|^2+|\xi|^2)+2(\l^2|y|^2-\hbar^2\Dlt_y)
$$
so that
$$
\ba
\iint_{\bR^d\times\bR^d}\Tr_\fH(Q(x,\xi)^{1/2}c(x,\xi,y,\hbar D_y)Q(x,\xi)^{1/2})dxd\xi
\\
\le 2\iint_{\bR^d\times\bR^d}\Tr_\fH(Q(x,\xi)^{1/2}(\l^2|x|^2+|\xi|^2)Q(x,\xi)^{1/2})dxd\xi
\\
+2\iint_{\bR^d\times\bR^d}\Tr_\fH(Q(x,\xi)^{1/2}(\l^2|y|^2-\hbar^2\Dlt_y)Q(x,\xi)^{1/2})dxd\xi&\,.
\ea
$$
First
$$
\ba
\iint_{\bR^d\times\bR^d}\Tr_\fH(Q(x,\xi)^{1/2}(\l^2|x|^2+|\xi|^2)Q(x,\xi)^{1/2})dxd\xi
\\
=
\iint_{\bR^d\times\bR^d}(\l^2|x|^2+|\xi|^2)\Tr_\fH(Q(x,\xi))dxd\xi
\\
=
\iint_{\bR^d\times\bR^d}(\l^2|x|^2+|\xi|^2)f(x,\xi)dxd\xi&\,.
\ea
$$

Since $R\in\cD^2(\fH)$, one has
$$
\ba
\Tr_\fH(R^{1/2}(\l^2|y|^2-\hbar^2\Dlt_y)R^{1/2})
\\
=\Tr_{\fH}((\l^2|y|^2-\hbar^2\Dlt_y)^{1/2}R(\l^2|y|^2-\hbar^2\Dlt_y)^{1/2})
\\
=\iint_{\bR^d\times\bR^d}\Tr_{\fH}((\l^2|y|^2-\hbar^2\Dlt_y)^{1/2}Q(x,\xi)dxd\xi(\l^2|y|^2-\hbar^2\Dlt_y)^{1/2})<\infty&\,,
\ea
$$
where the first equality is \eqref{FiniteEnerg2}, while the second follows from the monotone convergence theorem (Theorem 1.27 in \cite{RudinRCA}) applied to a spectral decomposition of the harmonic oscillator $\l^2|y|^2-\hbar^2\Dlt_y$.

In particular
$$
\Tr_{\fH}(\l^2|y|^2-\hbar^2\Dlt_y)^{1/2}Q(x,\xi)(\l^2|y|^2-\hbar^2\Dlt_y)^{1/2})<\infty
$$
for a.e. $(x,\xi)\in\bR^d\times\bR^d$. Applying Lemma \ref{L-FiniteEnerg} to $A=\l^2|y|^2-\hbar^2\Dlt_y$ and $T=Q(x,\xi)$ for a.e. $(x,\xi)\in\bR^d\times\bR^d$, one has
$$
\ba
\Tr_{\fH}((\l^2|y|^2-\hbar^2\Dlt_y)^{1/2}Q(x,\xi)(\l^2|y|^2-\hbar^2\Dlt_y)^{1/2})
\\
=
\Tr_{\fH}(Q(x,\xi)^{1/2}(\l^2|y|^2-\hbar^2\Dlt_y)Q(x,\xi)^{1/2})
\ea
$$
for a.e. $(x,\xi)\in\bR^d\times\bR^d$. Integrating both sides of this equality over $\bR^d\times\bR^d$, one finds that
$$
\ba
\iint_{\bR^d\times\bR^d}\Tr_{\fH}(Q(x,\xi)^{1/2}(\l^2|y|^2-\hbar^2\Dlt_y)Q(x,\xi)^{1/2})dxd\xi
\\
=
\Tr_{\fH}((\l^2|y|^2-\hbar^2\Dlt_y)^{1/2}R(\l^2|y|^2-\hbar^2\Dlt_y)^{1/2})<\infty&\,.
\ea
$$

In particular
$$
\Tr_{\fH}(Q(x,\xi)^{1/2}c(x,\xi,y,\hbar D_y)Q(x,\xi)^{1/2})<\infty
$$
for a.e. $(x,\xi)\in\bR^d\times\bR^d$. Applying again Lemma \ref{L-FiniteEnerg} with $A=c(x,\xi,y,\hbar D_y)$ and $T=Q(x,\xi)$ for all such $(x,\xi)$ shows that
$$
\ba
\Tr_{\fH}(Q(x,\xi)^{1/2}c(x,\xi,y,\hbar D_y)Q(x,\xi)^{1/2})
\\
=\Tr_{\fH}(c(x,\xi,y,\hbar D_y)^{1/2}Q(x,\xi)c(x,\xi,y,\hbar D_y)^{1/2})
\ea
$$
for a.e. $(x,\xi)\in\bR^d$, and the equality in the lemma follows from integrating both sides of this last identity over $\bR^d\times\bR^d$.
\end{proof}

\smallskip
The main properties of this pseudo-metric are recalled in the following theorem. Before stating it, we recall some fundamental notions and introduce some notations.

The Wigner transform of $R\in\cD(\fH)$ is 
$$
W_{\hbar}[R](x,\xi)=\tfrac1{(2\pi)^d}\int_{\bR^d}r(x+\tfrac12\hbar y,x-\tfrac12\hbar y)e^{-i\xi\cdot y}dy
$$
where $r$ is the integral kernel of $R$. Obviously $W_{\hbar}[R]$ is real-valued, but in general $W_{\hbar}[R]$ is not a.e. nonnegtive in general. 

Instead of the Wigner transform, one can consider a mollified variant thereof, the Husimi transform of $R$, that is
$$
\tilde W_{\hbar}[R](x,\xi)=(e^{\hbar\Dlt_{x,\xi}/4}W_{\hbar}[R])(x,\xi)\ge 0\hbox{ for a.e. }(x,\xi)\in\bR^d\times\bR^d\,.
$$

The Schr\"odinger coherent state is
$$
|q,p\ra(x):=(\pi\hbar)^{-d/4}e^{-|x-q|^2/2\hbar}e^{ip\cdot(x-q/2)/\hbar}\,.
$$
For each Borel probability measure $\mu$ on $\bR^d\times\bR^d$, one defines the T\"oplitz operator with symbol $(2\pi\hbar)^d\mu$:
$$
\OP^T_{\hbar}[(2\pi\hbar)^d\mu]:=\iint_{\bR^d\times\bR^d}|q,p\ra\la q,p|\mu(dqdp)\in\cD(\fH)\,.
$$

\begin{Prop}\lb{P-UpBound}
For each probability density $f$ and each Borel probability measure $\mu$ on $\bR^d\times\bR^d$ with finite second order moment \eqref{2ndMom}. Then
$$
\OP^T_{\hbar}[(2\pi\hbar)^d\mu]\in\cD^2(\fH)\,,
$$
and one has
$$
E_{\hbar,\l}(f,\OP^T_{\hbar}[(2\pi\hbar)^d\mu])^2\le\max(1,\l^2)\MKd(f,\mu)^2+\tfrac12(\l^2+1)d\hbar\,.
$$
\end{Prop}

\begin{proof}
Let $P(x,\xi,dqdp)$ be an optimal coupling of $f(x,\xi)$ and $\mu(dqdp)$ for $\MKd$. Set $Q(x,\xi):=\OP^T_{\hbar}[(2\pi\hbar)^dP(x,\xi,\cdot)]$. Then $Q\in\cC(f,\OP^T_{\hbar}[(2\pi\hbar)^d\mu])$ according to Lemma 3.1 in \cite{FGTPaul}), 
so that
$$
\ba
E_{\hbar,\l}(f,\OP^T_{\hbar}[(2\pi\hbar)^d\mu])^2
\\
\le\iint_{\bR^d\times\bR^d}\Tr_{\fH}(Q(x,\xi)^{1/2}c_\l(x,\xi,y,\hbar D_y)Q(x,\xi)^{1/2})dxd\xi\,.
\ea
$$

For each $p,q\in\bR^d$, one has
$$
\ba
\Tr_{\fH}(c_\l(x,\xi,y,\hbar D_y)^{1/2} |q,p\ra\la q,p| c_\l(x,\xi,y,\hbar D_y)^{1/2})
\\
=\la q,p|c_\l(x,\xi,y,\hbar D_y) |q,p\ra=\l^2|x-q|^2+|\xi-p|^2+\tfrac12(\l^2+1)\hbar
\ea
$$
according to fla. (55) in \cite{FGMouPaul}. For each finite positive Borel measure $m$ on $\bR^d\times\bR^d$, one has
$$
\ba
\Tr_{\fH}(c_\l(x,\xi,y,\hbar D_y)^{1/2}\OP^T_{\hbar}[(2\pi\hbar)^dm] c_\l(x,\xi,y,\hbar D_y)^{1/2})
\\
=\iint_{\bR^d\times\bR^d}(\l^2|x-q|^2+|\xi-p|^2+\tfrac12(\l^2+1)\hbar)m(dpdq)&\,.
\ea
$$
by the monotone convergence theorem (Theorem 1.27 in \cite{RudinRCA}) applied to a spectral decomposition of the transportation cost operator $c_\l(x,\xi,y,\hbar D_y)$, which is a shifted harmonic oscillator.

Specializing this formula to the case $x=\xi=0$ and $m=\mu$ shows that the operator $\OP^T_{\hbar}[(2\pi\hbar)^d\mu] \in\cD^2(\fH)$.

Specializing this formula to the case $m=P(x,\xi,dqdp)$ and integrating in $(x,\xi)$ shows that
$$
\ba
\iint_{\bR^d\times\bR^d}\Tr_{\fH}(c_\l(x,\xi,y,\hbar D_y)^{1/2}\OP^T_{\hbar}[(2\pi\hbar)^dP(x,\xi,\cdot)] c_\l(x,\xi,y,\hbar D_y)^{1/2})dxd\xi
\\
=\iint_{\bR^d\times\bR^d}\iint_{\bR^d\times\bR^d}(\l^2|x-q|^2+|\xi-p|^2)P(x,\xi,dqdp)+\tfrac12(\l^2+1)
\\
=\MKd(f,\mu)^2+\tfrac12(\l^2+1)\hbar
\ea
$$
and since $Q:\,(x,\xi)\mapsto\OP^T_{\hbar}[(2\pi\hbar)^dP(x,\xi,\cdot)]$ belongs to $\cC(f,\OP^T_{\hbar}[(2\pi\hbar)^d\mu])$, 
$$
\ba
\iint_{\bR^d\times\bR^d}\Tr_{\fH}(Q(x,\xi)^{1/2}c_\l(x,\xi,y,\hbar D_y)Q(x,\xi)^{1/2})dxd\xi
\\
=\iint_{\bR^d\times\bR^d}\Tr_{\fH}(c_\l(x,\xi,y,\hbar D_y)^{1/2}Q(x,\xi)c_\l(x,\xi,y,\hbar D_y)^{1/2})dxd\xi&\,.
\ea
$$
With the previous equality and the inequality above, the proof is complete.
\end{proof}


\section{Evolution of the pseudo-metric under the Schr\"odinger dynamics}\label{evolmet}


Denote by $t\mapsto(X(t;x,\xi),\Xi(t;x,\xi))$ the solution of the Cauchy problem for the Hamiltonian system
$$
\dot{X}=\Xi\,,\quad\dot{\Xi}=-\grad V(X)\,,\qquad (X(0;x,\xi),\Xi(0;x,\xi))=(x,\xi)\,.
$$
Since $V\in C^{1,1}(\bR^d)$, this solution is defined for all $t\in\bR$, for all $x,\xi\in\bR^d$. Henceforth, we denote by $\Phi_t$ the map $(x,\xi)\mapsto\Phi_t(x,\xi):=(X(t;x,\xi),\Xi(t;x,\xi))$, and by $H\equiv H(x,\xi):=\tfrac12|\xi|^2+V(x)$
the Hamiltonian.

On the other hand, assume that $V^-\in L^{d/2}(\bR^d)$, so that $\cH:=-\tfrac12\hbar^2\Dlt+V$ is self-adjoint on $\fH$ by Lemma 4.8b in chapter VI, \S 4 of \cite{Kato}. Then $U(t):=\exp(it\cH/\hbar)$ is a unitary group on $\fH$.

\begin{Thm}\lb{T-IneqFGTP}
Let $f^{in}$ be a probability density on $\bR^d\times\bR^d$ which satisfies \eqref{2ndMom}, and let $R^{in}\in\cD^2(\fH)$. For each $t\ge 0$, set
$$
R(t):=U(t)^*R^{in}U(t)\,,\quad f(t,X,\Xi):=f^{in}(\Phi_{-t}(X,\Xi))\quad\hbox{ for a.e. }(X,\Xi)\in\bR^d\times\bR^d\,.
$$
Then, for each $\l>0$ and each $t\ge 0$, one has
$$
E_{\hbar,\l}(f(t,\cdot,\cdot),R(t))\le E_{\hbar,\l}(f^{in},R^{in})\exp\left(\tfrac12t\left(\l+\frac{\Lip(\grad V)^2}\l\right)t\right)\,.
$$
\end{Thm}

This theorem is a slight improvement of Theorem 2.7 in \cite{FGTPaul} in the special case $N=1$. For the sake of being complete, we recall the argument in \cite{FGTPaul}, with the appropriate modifications.

\begin{proof}
Let $Q^{in}\in\cC(f^{in},R^{in})$. Set 
$$
Q(t,X,\Xi):=U(t)^*Q^{in}\circ\Phi_{-t}(X,\Xi)U(t)
$$
for all $t\in\bR$ and a.e. $(x,\xi)\in\bR^d\times\bR^d$, and
$$
\cE(t):=\iint_{\bR^{2d}}\Tr_{\fH}(Q(t,X,\Xi)^{1/2}c_\l(X,\Xi,y,\hbar D_y)Q(t,X,\Xi)^{1/2})dXd\Xi\,.
$$
Since $\Phi_t$ leaves the phase space volume element $dxd\xi$ invariant
$$
\cE(t)\!=\!\!\!\iint_{\bR^{2d}}\!\!\Tr_{\fH}(\sqrt{Q^{in}(x,\xi)}U(t)c_\l(\Phi_t(x,\xi),y,\hbar D_y)U(t)^*\sqrt{Q^{in}(x,\xi)})dx\xi\,.
$$

By construction, $Q(t,\cdot,\cdot)\in\cC(f(t,\cdot,\cdot),R(t))$. Indeed, for a.e. $(X,\Xi)\in\bR^d$, 
$$
0\le Q^{in}(\Phi_{-t}(X,\Xi))=Q^{in}(\Phi_{-t}(X,\Xi))^*\in\cL(\fH)
$$
so that $Q(t,X,\Xi)\in\cL(\fH)$ satisfies
$$
\ba
Q(t,X,\Xi)=&U(t)Q^{in}(\Phi_{-t}(X,\Xi))U(t)^*
\\
=&U(t)Q^{in}(\Phi_{-t}(X,\Xi))U(t)^*=Q(t,X,\Xi)^*\ge 0\,.
\ea
$$
Besides
$$
\Tr_\fH(Q(t,X,\Xi))=\Tr_\fH(Q^{in}(\Phi_{-t}(X,\Xi)))=f^{in}(\Phi_{-t}(X,\Xi))=f(t,X,\Xi)
$$
while
$$
\ba
\iint_{\bR^d\times\bR^d}Q(t,X,\Xi)dXd\Xi=U(t)\left(\iint_{\bR^d\times\bR^d}Q^{in}(\Phi_{-t}(X,\Xi))dXd\Xi\right)U(t)^*
\\
=U(t)\left(\iint_{\bR^d\times\bR^d}Q^{in}(x,\xi)dxd\xi\right)U(t)^*=U(t)R^{in}U(t)^*=R(t)&\,.
\ea
$$
In particular
$$
\cE(t)\ge E_{\hbar,\l}(f(t),R(t))\,,\qquad\hbox{ for each }t\ge 0\,.
$$

Let $e_j(x,\xi,\cdot)$ for $j\in\bN$ be a $\fH$-complete orthonormal system of eigenvectors of $Q^{in}(x,\xi)$ for a.e. $x,\xi\in\bR^d$. Hence
$$
\ba
\Tr_{\fH}(\sqrt{Q^{in}(x,\xi)}U(t)c_\l(\Phi_t(x,\xi),y,\hbar D_y)U(t)^*\sqrt{Q^{in}(x,\xi)})
\\
=\sum_{j\in\bN}\rho_j(x,\xi)\la U(t)e_j(x,\xi)|c_\l(\Phi_t(x,\xi),y,\hbar D_y)|U(t)e_j(x,\xi)\ra
\ea
$$
where $\rho_j(x,\xi)$ is the eigenvalue of $Q^{in}(x,\xi)$ defined by
$$
Q^{in}(x,\xi)e_j(x,\xi)=\rho_j(x,\xi)e_j(x,\xi)\,,\quad\hbox{ for a.e. }(x,\xi)\in\bR^d\times\bR^d\,.
$$

If $\phi\equiv\phi(y)\in C^\infty_c(\bR^d)$, the map
$$
t\mapsto\la U(t)\phi|c_\l(\Phi_t(x,\xi),y,\hbar D_y)|U(t)\phi\ra
$$
is of class $C^1$ on $\bR$, and one has
$$
\ba
\frac{d}{dt}\la U(t)\phi|c_\l(\Phi_t(x,\xi),y,\hbar D_y)|U(t)\phi\ra
\\
=\La\frac{i}{\hbar}\cH U(t)\phi\Big|c_\l(\Phi_t(x,\xi),y,\hbar D_y)\Big|U(t)\phi\Ra
\\
+\La U(t)\phi\Big|c_\l(\Phi_t(x,\xi),y,\hbar D_y)\Big|\frac{i}{\hbar}\cH U(t)\phi\Ra
\\
+\la U(t)\phi|\{H(\Phi_t(x,\xi)),c_\l(\Phi_t(x,\xi),y,\hbar D_y)\}|U(t)\phi\ra&\,.
\ea
$$
In other words
$$
\ba
\frac{d}{dt}\la U(t)\phi|c_\l(\Phi_t(x,\xi),y,\hbar D_y)|U(t)\phi\ra
\\
=\La U(t)\phi\Big|\frac{i}\hbar[\cH,c_\l(\Phi_t(x,\xi),y,\hbar D_y)]\Big|U(t)\phi\Ra
\\
+\la U(t)\phi|\{H(\Phi_t(x,\xi)),c_\l(\Phi_t(x,\xi),y,\hbar D_y)\}|U(t)\phi\ra&\,.
\ea
$$

A straightforward computation shows that
$$
\ba
\{H(\Phi_t(x,\xi)),c_\l(\Phi_t(x,\xi),y,\hbar D_y)\}+\frac{i}\hbar[\cH,c_\l(\Phi_t(x,\xi),y,\hbar D_y)]
\\
=\l^2\sum_{k=1}^d\left((X_k-y_k)(\Xi_k-\hbar D_{y_k})+(\Xi_k-\hbar D_{y_k})(X_k-y_k)\right)
\\
-\sum_{k=1}^d\left((\d_kV(X)-\d_kV(y))(\Xi_k-\hbar D_{y_k})+(\Xi_k-\hbar D_{y_k})(\d_kV(X)-\d_kV(y))\right)
\\
\le\l\sum_{k=1}^d(\l^2|X_k\!-\!y_k|^2\!+\!|\Xi_k\!-\!\hbar D_{y_k}|^2)\!+\!\frac1\l\sum_{k=1}^d\left(\l^2|\d_kV(X)\!-\!\d_kV(y)|^2\!+\!|\Xi_k\!-\!\hbar D_{y_k}|^2\right)
\\
\le\l\sum_{k=1}^d(\l^2|X_k-y_k|^2+|\Xi_k-\hbar D_{y_k}|^2)+\frac{\Lip(\grad V)^2}\l\sum_{k=1}^d\left(\l^2|X_k-y|^2+|\Xi_k-\hbar D_{y_k}|^2\right)
\\
\le\left(\l+\frac{\Lip(\grad V)^2}\l\right)c_\l(X,\Xi,y,\hbar D_y)&\,.
\ea
$$
Hence
$$
\ba
\la U(t)\phi|c_\l(\Phi_t(x,\xi),y,\hbar D_y)|U(t)\phi\ra\le\la\phi|c_\l(x,\xi,y,\hbar D_y)|\phi\ra
\\
+\left(\l+\frac{\Lip(\grad V)^2}\l\right)\int_0^t\la U(s)\phi|c_\l(\Phi_s(x,\xi),y,\hbar D_y)|U(s)\phi\ra ds
\ea
$$
so that
$$
\la U(t)\phi|c_\l(\Phi_t(x,\xi),y,\hbar D_y)|U(t)\phi\ra\le\la\phi|c_\l(x,\xi,y,\hbar D_y)|\phi\ra\exp\left(\left(\l+\frac{\Lip(\grad V)^2}\l\right)t\right)
$$
for each $\phi\in C^\infty_c(\bR^d)$. By density of $C^\infty_c(\bR^d)$ in the form domain of $c_\l(x,\xi,y,\hbar D_y)$
$$
\ba
0\le\la U(t)e_j(x,\xi)|c_\l(\Phi_t(x,\xi),y,\hbar D_y)|U(t)e_j(x,\xi)\ra
\\
\le\la e_j(x,\xi)|c_\l(x,\xi,y,\hbar D_y)|e_j(x,\xi)\ra\exp\left(\left(\l+\frac{\Lip(\grad V)^2}\l\right)t\right)
\ea
$$
for a.e. $(x,\xi)\in\bR^d\times\bR^d$, so that
$$
\ba
\Tr_{\fH}(\sqrt{Q^{in}(x,\xi)}U(t)c_\l(\Phi_t(x,\xi),y,\hbar D_y)U(t)^*\sqrt{Q^{in}(x,\xi)})
\\
=\sum_{j\in\bN}\rho_j(x,\xi)\la U(t)e_j(x,\xi)|c_\l(\Phi_t(x,\xi),y,\hbar D_y)|U(t)e_j(x,\xi)\ra
\\
\le\exp\left(\left(\l+\frac{\Lip(\grad V)^2}\l\right)t\right)\sum_{j\in\bN}\rho_j(x,\xi)\la e_j(x,\xi)|c_\l(x,\xi,y,\hbar D_y)|e_j(x,\xi)\ra
\\
=\exp\left(\left(\l+\frac{\Lip(\grad V)^2}\l\right)t\right)\Tr_{\fH}(\sqrt{Q^{in}(x,\xi)}c_\l(x,\xi,y,\hbar D_y)\sqrt{Q^{in}(x,\xi)})&\,.
\ea
$$
Integrating both side of this inequality over $\bR^d\times\bR^d$ shows that
$$
\cE(t)\le\cE(0)\exp\left(\left(\l+\frac{\Lip(\grad V)^2}\l\right)t\right)\,.
$$

Hence, for each $t\ge 0$ and each $Q^{in}\in\cC(f,R)$, one has
$$
E_{\hbar,\l}(f(t),R(t))^2\le\cE(0)\exp\left(\left(\l+\frac{\Lip(\grad V)^2}\l\right)t\right)\,.
$$
Minimizing the right hand side of this inequality as $Q^{in}$ runs through $\cC(f^{in},R^{in})$, one arrives at the inequality
$$
E_{\hbar,\l}(f(t),R(t))\le E_{\hbar,\l}(f^{in},R^{in})\exp\left(\tfrac12\left(\l+\frac{\Lip(\grad V)^2}\l\right)t\right)\,.
$$
\end{proof}


\section{The observation inequality}\label{observmet}


In this section, we state and prove an observation inequality for the Schr\"odinger equation.

Let $K$ be a compact subset of $\bR^d\times\bR^d$, let $\Om$ be an open set of $\bR^d$ and let $T>0$. We recall the ``geometric condition'' \`a la Bardos-Lebeau-Rauch \cite{BLR} for this problem:
$$
\hbox{for each }(x,\xi)\in K\,,\hbox{ there exists }t\in(0,T)\hbox{ s.t. }X(t;x,\xi)\in\Om\,.\leqno{\hbox{(GC)}}
$$

 \begin{Thm}\lb{T-Obs}
Assume that $V$ belongs to $C^{1,1}(\bR^d)$ and that $V^-\in L^{d/2}(\bR^d)$. Let $T>0$, let $K\subset\bR^d\times\bR^d$ be compact and let $\Om\subset\bR^d$ be an open set of $\bR^d$ satisfying (GC). Let $\chi\in\Lip(\bR^d)$ 
be such that $\chi(x)>0$ for each $x\in\Om$.


For each $t\ge 0$, set
$$
R(t):=U(t)^*R^{in}U(t)\,,\quad f(t,X,\Xi):=f^{in}(\Phi_{-t}(X,\Xi))\quad\hbox{ for a.e. }(X,\Xi)\in\bR^d\times\bR^d\,.
$$
Then, when $R^{in}$ is a pure state $|\psi^{in}\rangle\langle\psi^{in}|$,
$$
\ba
\int_0^T
\int_{\bR^d}\chi(x)|\psi(t,x)|^2dx)dt\ge&
\inf_{(x,\xi)\in K}\int_0^T\chi(X(t;x,\xi))dt
\iint_{(x,\xi)\in K}\hus{\psi^{in}}(x,\xi)dxd\xi\nn\\
\\
&-4{\Lip(\chi)}\frac{\exp\left(\tfrac12\left(1+{\Lip(\grad V)^2}\right)T\right)-1}{\tfrac12\left(1+{\Lip(\grad V)^2}\right)}
\Delta(\psi^{in}).
\ea
$$
When $R^{in}:=\OP^T[(2\pi\hbar)^df^{in}]$ is a T\"oplitz operator of symbol   a probability density $f^{in}$ on $\bR^d\times\bR^d$ with support in $K$,  
$$
\ba
\int_0^T\Tr(\chi R(t))dt\ge&\inf_{(x,\xi)\in K}\int_0^T\chi(X(t;x,\xi))dt
\\
&-\Lip(\chi)C(T,\Lip(\grad V))\sqrt{2d\hbar}
\ea
$$
where
$$
C(T,L)=\inf_{\l>0}\frac{\exp\left(\tfrac12\left(\l+\frac{L^2}\l\right)T\right)-1}{\left(\l+\frac{L^2}{\l}\right)}\sqrt{1+\frac1{\l^2}}\,.
$$
In particular, setting $\l=L$
$$
C(T,L)\le\frac{e^{LT}-1}{2L}\sqrt{1+\frac1{L^2}}\,.
$$
\end{Thm}

\smallskip
In fact, one can eliminate all mention of the cutoff function $\chi$ in the final statement, as follows.

\begin{Cor}\label{C-Obs}
Under the same assumptions as in Theorem \ref{T-Obs}, one has
$$
C[T,K,\Om]:=\inf_{(x,\xi)\in K}\int_0^T\indc_{\Om}(X(t;x,\xi))dt>0\,,
$$
and for each $\de>0$, denoting $
\Om_\de:=\{x\in\bR^d\,|\,\mathrm{dist}(x,\Om)<\de\}\,.
$,
$$
\int_0^T\Tr(\indc_{\Om_\de}R(t))dt\ge C[T,K,\Om]-C(T,\Lip(\grad V))\frac{\sqrt{2d\hbar}}{\de}\,
$$
in the T\"oplitz case, and 
$$
\ba
\int_0^T
\int_{\Omega_\delta}|\psi(t,x)|^2dx)dt\ge&
\inf_{(x,\xi)\in K}\int_0^T\indc_\Omega(X(t;x,\xi))dt
\iint_{(x,\xi)\in K}\hus{\psi^{in}}(x,\xi)dxd\xi\nn\\
\\
&-4\frac{\exp\left(\tfrac12\left(1+{\Lip(\grad V)^2}\right)T\right)-1}{\tfrac12\left(1+{\Lip(\grad V)^2}\right)}
\frac{\Delta(\psi^{in})}\delta
\ea
$$
in the pure state case.
\end{Cor}

\smallskip
The corollary can be used to obtain an observation inequality for T\"oplitz operators as ``test observables" as follows: let $T>0$ be an observation time, let $K\subset\bR^d\times\bR^d$ be a compact subset of the phase-space supporting the initial data, and let $\Om\subset\bR^d$
be the open set where one observes the solution of the Schr\"odinger equation on the time interval $[0,T]$. Assume that $T,K,\Om$ satisfies the geometric condition (GC). With these data, one computes $C[T,K,\Om]>0$. Choose then 
$\hbar,\de>0$ so that
$$
\frac{\hbar}{\de^2}<\frac{C[T,K,\Om]^2}{2dC(T,\Lip(\grad V))^2}\,.
$$
Then the Heisenberg equation \eqref{Heisen} satisfies the observability property on $[0,T]\times\Om_\de$ for all T\"oplitz initial density operators whose symbol is supported in $K$.

\begin{proof}[Proof of the corollary]
Since $\Om$ is open, the function $\indc_\Om$ is lower semicontinuous. According to condition (GC), for each $(x,\xi)\in K$, there exists $t_{x,\xi}\in(0,T)$ such that $\indc_{\Om}(X(t_{x,\xi};x,\xi))=1$. Since the set
$$
\{t\in(0,T)\,|\,\indc_{\Om}(X(t;x,\xi))>1/2\}
$$ 
is open, there exists $\eta_{x,\xi}>0$ such that 
$$
[t_{x,\xi}-\eta_{x,\xi},[t_{x,\xi}+\eta_{x,\xi}]\subset(0,T)
$$
and then
$$
\int_0^T\indc_{\Om}(X(t;x,\xi))dt\ge 2\eta_{x,\xi}>0\,,\quad\hbox{ for each }(x,\xi)\in K\,.
$$
By Fatou's lemma, the function
$$
(x,\xi)\mapsto\int_0^T\indc_{\Om}(X(t;x,\xi))dt
$$
is lower semicontinuous, and positive on $K$. Hence
$$
C[T,K,\Om]:=\inf_{(x,\xi)\in K}\int_0^T\indc_{\Om}(X(t;x,\xi))dt>0\,.
$$

Apply Theorem \ref{T-Obs} with $\chi$ defined as follows:
$$
\chi_\de(x)=\left(1-\frac{\hbox{dist}(x,\Om)}{\de}\right)_+\,,\quad\hbox{ in which case }\Lip(\chi)=\frac1\de\,.
$$
One concludes by observing that
$$
\int_0^T\Tr(\chi_\de R(t))dt\int_0^T\Tr(\indc_{\Om_\de}R(t))dt\,,
$$
whereas
$$
\int_0^T\indc_{\Om}(X(t;x,\xi))dt\le\int_0^T\chi_\de(X(t;x,\xi))dt\,.
$$
\end{proof}

\begin{proof}
Notice that
$$
\ba
\Tr(\chi(R(t))-\iint_{\bR^d\times\bR^d}\chi(x)f(t,x,\xi)dxd\xi
\\
=
\iint_{\bR^d\times\bR^d}\Tr_\fH((\chi(y)-\chi(x))Q(t,x,\xi))dxd\xi
\ea
$$
for each $Q\equiv Q(t,x,\xi)\in\cC(f(t),R(t))$. Hence
$$
\ba
\left|\Tr(\chi R(t))-\iint_{\bR^d\times\bR^d}\chi(x)f(t,x,\xi)dxd\xi\right|
\\
=\left|\iint_{\bR^d\times\bR^d}\Tr_\fH((\chi(y)-\chi(x))Q(t,x,\xi))dxd\xi\right|
\\
\le\iint_{\bR^d\times\bR^d}|\Tr_\fH((\chi(y)-\chi(x))Q(t,x,\xi))|dxd\xi
\\
=\iint_{\bR^d\times\bR^d}|\Tr_\fH(Q(t,x,\xi)^{1/2}(\chi(y)-\chi(x))Q(t,x,\xi)^{1/2})|dxd\xi
\\
\le\iint_{\bR^d\times\bR^d}\Tr_\fH(Q(t,x,\xi)^{1/2}|\chi(y)-\chi(x)|Q(t,x,\xi)^{1/2})dxd\xi
\\
\le\Lip(\chi)\iint_{\bR^d\times\bR^d}\Tr_\fH(Q(t,x,\xi)^{1/2}|x-y|Q(t,x,\xi)^{1/2})dxd\xi
\\
\le\Lip(\chi)\iint_{\bR^d\times\bR^d}\Tr_\fH\left(Q(t,x,\xi)^{1/2}\tfrac12\left(\eps|x-y|^2+\frac1\eps\right)Q(t,x,\xi)^{1/2}\right)|dxd\xi&\,.
\ea
$$
Minimizing in $\eps>0$ shows that
$$
\ba
\left|\Tr_\fH(\chi R(t))-\iint_{\bR^d\times\bR^d}\chi(x)f(t,x,\xi)dxd\xi\right|
\\
\le
\Lip(\chi)\left(\iint_{\bR^d\times\bR^d}\Tr_\fH\left(Q(t,x,\xi)^{1/2}|x-y|^2Q(t,x,\xi)^{1/2}\right)|dxd\xi\right)^{1/2}
\\
\le\frac{\Lip(\chi)}{\l}\left(\iint_{\bR^d\times\bR^d}\Tr_\fH\left(Q(t,x,\xi)^{1/2}c_\l(x,\xi,y,\hbar D_y)Q(t,x,\xi)^{1/2}\right)|dxd\xi\right)^{1/2}&\,.
\ea
$$
This holds for each $Q(t)\in\cC(f(t),R(t))$; minimizing in $Q(t)\in\cC(f(t),R(t))$ leads to the bound
$$
\left|\Tr_\fH(\chi R(t))-\iint_{\bR^d\times\bR^d}\chi(x)f(t,x,\xi)dxd\xi\right|\le\frac{\Lip(\chi)}{\l}E_{\hbar,\l}(f(t),R(t))\,.
$$
By Theorem \ref{T-IneqFGTP}
$$
\ba
\left|\Tr_\fH(\chi R(t))-\iint_{\bR^d\times\bR^d}\chi(x)f(t,x,\xi)dxd\xi\right|
\\
\le\frac{\Lip(\chi)}{\l}E_{\hbar,\l}(f^{in},R^{in})\exp\left(\tfrac12\left(\l+\frac{\Lip(\grad V)^2}\l\right)t\right)&\,.
\ea
$$

On the other hand
$$
\ba
\iint_{\bR^d\times\bR^d}\chi(x)f(t,x,\xi)dxd\xi=&\iint_{\bR^d\times\bR^d}\chi(x)f^{in}(X(t;x,\xi),\Xi(t;x,\xi))dxd\xi
\\
=&\iint_{\bR^d\times\bR^d}\chi(X(t;x,\xi))f^{in}(x,\xi)dxd\xi\,.
\ea
$$

Hence

\begin{eqnarray}
\int_0^T\Tr(\chi R(t))dt&\ge&\iint_{\bR^d\times\bR^d}\left(\int_0^T\chi(X_t(x,\xi))dt\right)f^{in}(x,\xi)dxd\xi\nn
\\
&&-\frac{\Lip(\chi)}{\l}E_{\hbar,\l}(f^{in},R^{in})\int_0^T\exp\left(\tfrac12\left(\l+\frac{\Lip(\grad V)^2}\l\right)t\right)dt\nn
\\
&\ge&\iint_{\bR^d\times\bR^d}\left(\int_0^T\chi(X_t(x,\xi))dt\right)f^{in}(x,\xi)dxd\xi\nn
\\
&&-\frac{\Lip(\chi)}{\l}\frac{\exp\left(\tfrac12\left(\l+\frac{\Lip(\grad V)^2}\l\right)T\right)-1}{\tfrac12\left(\l+\frac{\Lip(\grad V)^2}\l\right)}E_{\hbar,\l}(f^{in},R^{in})\,.\nn\\
&\geq&
\inf_{(x,\xi)\in K}\int_0^T\chi(X(t;x,\xi))dt
\iint_{(x,\xi)\in K}f^{in}(x,\xi)dxd\xi\label{eqfin}\\
&&-\frac{\Lip(\chi)}{\l}\frac{\exp\left(\tfrac12\left(\l+\frac{\Lip(\grad V)^2}\l\right)T\right)-1}{\tfrac12\left(\l+\frac{\Lip(\grad V)^2}\l\right)}
E_{\hbar,\l}(f^{in},R^{in})\,.\nn
\end{eqnarray}

In particular, putting $f^{in}=\hus{R^{in}}$ and $\lambda=1$, one obtains
\begin{eqnarray}
\int_0^T\Tr(\chi R(t))dt&\ge&
\inf_{(x,\xi)\in K}\int_0^T\chi(X(t;x,\xi))dt
\iint_{(x,\xi)\in K}(\hus{R^{in}}(x,\xi)dxd\xi\nn\\
&&-\frac{\Lip(\chi)}{\l}\frac{\exp\left(\tfrac12\left(\l+\frac{\Lip(\grad V)^2}\l\right)T\right)-1}{\tfrac12\left(\l+\frac{\Lip(\grad V)^2}\l\right)}
E_{\hbar,\l}((\hus{R^{in}},R^{in})\,.\label{obsgen}
\end{eqnarray}

For $R^{in}=|\psi^{in}\rangle\langle\psi^{in}|$, we know by Proposition 9.1. in \cite{GPsemic} that $
E_{\hbar,1}(\hus{R^{in}},R^{in})\leq 2\Delta(R^{in})$ and
we get the conclusion of Theorem \ref{T-Obs} in the pure state case.

\vskip 1cm
If  $f^{in}$ is any compactly supported probability density, the inequality \eqref{eqfin} that
$$
\ba
\int_0^T\Tr(\chi R(t))dt\ge&\inf_{(x,\xi)\in\Supp(f^{in})}\int_0^T\chi(X(t;x,\xi))dt
\\
&-\frac{\Lip(\chi)}{\l}\frac{\exp\left(\tfrac12\left(\l+\frac{\Lip(\grad V)^2}\l\right)T\right)-1}{\tfrac12\left(\l+\frac{\Lip(\grad V)^2}\l\right)}E_{\hbar,\l}(f^{in},R^{in})\,.
\ea
$$

Now, if $R^{in}$ is the T\"oplitz operator with symbol $(2\pi\hbar)^d\mu^{in}$, where $\mu^{in}$ is a Borel probability measure on $\bR^d\times\bR^d$, 
$$
\ba
\int_0^T\Tr(\chi R(t))dt\ge\inf_{(x,\xi)\in\Supp(f^{in})}\int_0^T\chi(X(t;x,\xi))dt
\\
-\frac{\Lip(\chi)}{\l}\frac{\exp\left(\tfrac12\left(\l\!+\!\frac{\Lip(\grad V)^2}\l\right)T\right)\!-\!1}{\tfrac12\left(\l\!+\!\frac{\Lip(\grad V)^2}\l\right)}\sqrt{\max(1,\l^2)\MKd(f^{in}\!,\mu^{in})^2\!+\!\tfrac12(\l^2\!+\!1)d\hbar}&\,.
\ea
$$

In particular, if $R^{in}=\OP^T_{\hbar}[(2\pi\hbar)^df^{in}]$, one has
$$
\ba
\int_0^T\Tr(\chi R(t))dt\ge&\inf_{(x,\xi)\in\Supp(f^{in})}\int_0^T\chi(X(t;x,\xi))dt
\\
&-\frac{\Lip(\chi)}{\l}\frac{\exp\left(\tfrac12\left(\l\!+\!\frac{\Lip(\grad V)^2}\l\right)T\right)\!-\!1}{\tfrac12\left(\l\!+\!\frac{\Lip(\grad V)^2}\l\right)}\sqrt{\tfrac12(\l^2\!+\!1)d\hbar}\,.
\ea
$$
Maximizing the right hand side as $\l$ runs through $(0,+\infty)$, one finds that
$$
\ba
\int_0^T\Tr(\chi R(t))dt\ge&\inf_{(x,\xi)\in\Supp(f^{in})}\int_0^T\chi(X(t;x,\xi))dt
\\
&-\Lip(\chi)C(T,\Lip(\grad V))\sqrt{2d\hbar}\,,
\ea
$$
where
$$
C(T,L):=\inf_{\l>0}\frac{\exp\left(\tfrac12\left(\l\!+\!\frac{L^2}\l\right)T\right)\!-\!1}{\l^2+L^2}\sqrt{\l^2\!+\!1}\,.
$$
If $L>0$, one can take $\l=L$ so that
$$
C(T,L)\le\frac{e^{LT}-1}{2L^2}\sqrt{1+L^2}\,.
$$
\end{proof}

\smallskip
Notice that, in the case where $L=0$, one can choose $\l=2r/T$ with
$$
re^r=2(e^r-1)\,,\quad r>0\,,\qquad\l=2r/T\,,
$$
and find that
$$
C(T,0)\le\frac{e^r-1}{4r^2}T^2\sqrt{1+\frac{4r^2}{T^2}}\,.
$$
\vskip 1cm
\textbf{Acknowledgments.} We would like to thank warmly Claude Bardos for having read  the first version of this paper and mentioned several references.

\end{document}